\documentclass[11pt]{article}
\usepackage[margin=1.45in]{geometry}

\usepackage{mdwlist}
\usepackage{enumerate}

\usepackage{amsmath,amssymb}
\usepackage{graphics, subfigure, float}
\usepackage{fp, calc}
\usepackage{hyperref}
\usepackage{url}

\usepackage[T1]{fontenc} 
\usepackage{fourier}
\usepackage{bm}

\usepackage{amscd,amsthm}

\usepackage{pst-all}
\usepackage{pstricks-add}
\usepackage{pst-func}
\newpsobject{showgrid}{psgrid}{subgriddiv=1,griddots=10,gridlabels=6pt}
\usepackage{verbatim, comment}
\usepackage{datetime}

\newtheoremstyle{theorem}{1em}{1em}{\slshape}{0pt}{\bfseries}{.}{ }{}
\theoremstyle{theorem}
\newtheorem{theorem}{Theorem}

\newtheorem{conjecture}{Conjecture}
\newtheorem*{theorem*}{Theorem}

\newtheorem{lemma}[theorem]{Lemma}

\newtheorem*{claim*}{Claim}

\theoremstyle{remark}
\newtheorem{remark}{Remark}
\newtheorem*{remark*}{Remark}

\newtheoremstyle{example}{1em}{1em}{}{0pt}{\bfseries}{.}{ }{}

\providecommand{\setR}{\mathbb{R}}

\psset{linewidth=1pt, arrowsize=6pt}

\usepackage{calrsfs} 
\DeclareMathAlphabet{\pazocal}{OMS}{zplm}{m}{n}

\usepackage[displaymath,textmath,graphics, subfigure, floats]{preview} 
\PreviewEnvironment{center} 

\makeatother

\title{Balancing Polynomials in the Chebyshev Norm}
\begin{document}
\date{}
\author{Victor Reis \thanks{University of Washington, Seattle. Email: {\tt voreis@uw.edu}.}}
\maketitle
\begin{abstract}

Given $n$ polynomials $p_1, \dots, p_n$ of degree at most 
 $n$ with $\|p_i\|_\infty \le 1$ for $i \in [n]$, we show there exist signs $x_1, \dots, x_n \in \{-1,1\}$ so that 
\[\Big\|\sum_{i=1}^n x_i p_i\Big\|_\infty < 30\sqrt{n}, \]
where $\|p\|_\infty := \sup_{|x| \le 1} |p(x)|$. This result extends the Rudin-Shapiro sequence, which gives an upper bound of $O(\sqrt{n})$ for the Chebyshev polynomials $T_1, \dots, T_n$, and can be seen as a polynomial analogue of Spencer's "six standard deviations" theorem.

\end{abstract}

\section{Introduction}

Given a univariate polynomial $p : \setR \to \setR$, the \textit{Chebyshev norm} $\|p\|_\infty := \sup_{|x| \le 1} |p(x)|$ was introduced in 1854 by Chebyshev \cite{Cheb-1854} who also defined the \textit{Chebyshev polynomials} $T_n$ with $\|T_n\|_\infty = 1$ which satisfy $T_n(\cos \theta) = \cos(n\theta)$ for all $\theta \in \setR$.

The \textit{Rudin-Shapiro sequence}, independently discovered by Rudin \cite{10.2307/2033608} and Shapiro \cite{Shap-1952} provides signs $x_i \in \{-1,1\}$ with the property that for any positive integer $n$ and any complex number $z$ on the unit circle, we have the upper bound $|\sum_{i=1}^n x_i z^i| = O(\sqrt{n})$. By looking at the real part, this readily implies $\|\sum_{i=1}^n x_i T_i\|_\infty = O(\sqrt{n})$.

 Here we prove the following generalization:

\begin{theorem}
Given $n$ polynomials $p_1, \dots, p_n$ of degree at most $n$ with $\|p_i\|_\infty \le 1$ for $i \in [n]$, there exist signs $x_1, \dots, x_n \in \{-1,1\}$ so that 
\[\Big\|\sum_{i=1}^n x_i p_i\Big\|_\infty < 30\sqrt{n}. \]
 
\end{theorem}

We also show a matching lower bound using precisely the Chebyshev polynomials.
\begin{theorem}
For any choice of signs $x_1, \dots, x_n \in \{-1,1\}$ we have 
\[\Big\|\sum_{i=1}^n x_i T_i \Big\|_\infty \ge \sqrt{n/2}.\]
\end{theorem}
We would like to remark that for i.i.d. \textit{random signs} $x_1, \dots, x_n \sim \{-1,1\}$, the work of Salem and Zygmund implies an upper bound of $O(\sqrt{n \log n})$ at least for bounded trigonometric polynomials \cite{salem1954}. Another line of work relating polynomials and discrepancy theory culminated on the recent construction of flat Littlewood polynomials \cite{balister2019flat}.

A notoriously difficult conjecture in discrepancy theory is that of Koml\'os \cite{Spencer1985SixSD}:

\begin{conjecture} There exists a constant $C$ so that for any $\bm{v}_1, \dots, \bm{v}_n \in \setR^n$ with $\|\bm{v}_i\|_2 \le 1$ for $i \in [n]$, we can find signs $x_1, \dots, x_n \in \{-1,1\}$ with the property that $\|\sum_{i=1}^n x_i \bm{v}_i\|_\infty \le C$.
\end{conjecture}

We show that the Koml\'os conjecture implies a similar result with a weaker assumption:

\begin{theorem}
Let $p_1, \dots, p_n$ be polynomials of degree less than $n$ with $\int_{0}^\pi p_i^2 (\cos \theta) \ \mathrm{d} \theta \le 1$ for $i \in [n]$. Suppose the Koml\'os conjecture holds with constant $C$. Then there exist signs $x_1, \dots, x_n \in \{-1,1\}$ so that 
\[\Big\|\sum_{i=1}^n x_i p_i\Big\|_\infty < 3C\sqrt{n}. \]
\end{theorem}

\section{Background}

For the proof of Theorem 1 we will need a classical inequality of Bernstein \cite{lorentz1960}:
\begin{theorem}
Given a trigonometric polynomial $p(x) := \sum_{j=0}^n a_j \cos(jx)$ with $|p(x)| \le 1$ for all $x \in \setR$, we have $|p'(x)| \le n$ for all $x \in \setR$.
\end{theorem}

We will also need the following well-known equivalence between polynomials in $\cos x$ and trigonometric polynomials:

\begin{lemma}
For odd $n$, we have $\cos^n(x)= \frac{1}{2^{n-1}} \sum_{k=0}^{(n-1)/2} {n \choose k} \cos((n-2k) x)$, and for even $n$ we have $\cos^n(x) = \frac{1}{2^n} {n \choose {n/2}} + \frac{1}{2^{n-1}} \sum_{k=0}^{n/2 - 1} {n \choose k} \cos((n-2k)x)$. Thus any degree $n$ polynomial in $\cos x$ is also a trigonometric polynomial and satisfies Bernstein's inequality.
\end{lemma}

Finally, we use Spencer's theorem \cite{Spencer1985SixSD} to bound the $\ell_\infty$-discrepancy of vectors:

\begin{theorem}
Given vectors $\bm{v}_1, \dots, \bm{v}_n \in \setR^m$ with $\|\bm{v}_i\|_\infty \le 1$ there exist signs $x_1, \dots, x_n \in \{-1,1\}^n$ with $\|\sum_{i=1}^n x_i \bm{v}_i\|_\infty < 6\sqrt{m}$ when $m \ge n$, and we also have the upper bound $O(\sqrt{n \log(2m/n)})$.
\end{theorem}

\section{Proofs of main theorems}

The key technical lemma in the proof is the following characterization of polynomials with constant Chebyshev norm. 

\begin{lemma}
Suppose $p$ is a polynomial of degree at most $n$ such that $|p(\cos(\theta_0 + k\pi/9n))| \le 1$ for $k \in \{0, \dots, 9n-1\}$, for some fixed $\theta_0 \in [0,\pi/9n]$. Then in fact $\|p\|_\infty < 5/3$.
\end{lemma}

\begin{proof}
Denote $q(\theta) := p(\cos (\theta))$ and note that $\|p\|_\infty = \sup_{\theta \in [0, \pi]} |q(\theta)|$. By the extreme value theorem, this supremum is attained at some $\theta^* \in [0,\pi]$. Then either $\theta^* - \theta_0 \in [k\pi/9n, (k+1) \pi/9n]$ for some $k \in \{0, \dots, 9n-1\}$ or else $\theta^* \in [0, \theta_0]$. In the first case, Bernstein's inequality gives 
\begin{align*}|q(\theta^*)| - 1 & \le |q(\theta^*)| - |q(\theta_0 + k\pi/9n)| \\ & \le |q(\theta^*) - q(\theta_0 + k\pi/9n)| \\ & \le (\theta^* - \theta_0 - k\pi/9n) \cdot \sup_{\gamma \in [\theta_0 + \pi k/9n,\ \theta^*]} |q'(\gamma)| \\ & \le (\theta^* - \theta_0 - k\pi/9n) \cdot n \cdot |q(\theta^*)| \\ & \le \frac{\pi}{9} \cdot |q(\theta^*)|,  \end{align*}
so that $\|p\|_\infty = |q(\theta^*)| < 5/3$. Here we can indeed use Bernstein's inequality on $q$ due to Lemma 4. The case $\theta^* \in [0, \theta_0]$ follows similarly from $\theta_0 \le \pi/9n$.
\end{proof}

\begin{proof}[Proof of Theorem 1]
For each polynomial $p_i$ we may associate a vector $\bm{v}_i \in \setR^{9n}$ with $v_{ij} := p_i(\cos((j-1) \pi/9n))$. By the assumption $\|p_i\|_\infty \le 1$ we get $\|\bm{v}_i\|_\infty \le 1$ for all $i \in [n]$. Spencer's theorem gives signs $x_1, \dots, x_n \in \{-1,1\}$ such that the polynomial $q := \sum_{i=1}^n x_i p_i$ satisfies $|q(\cos (k \pi/9n))| \le 6\sqrt{9n}$ for all $k \in \{0,\dots,9n-1\}$.

Applying Lemma 7 to $q/(18\sqrt{n})$ with $\theta_0 := 0$ yields $\|q\|_\infty < 30\sqrt{n}$.
\end{proof}

\begin{remark}
For polynomials of degree at most $d \ge n$, we also get the more general upper bound of $O(\sqrt{n \log(2d/n)})$. Further, since Spencer's theorem has recently been made algorithmic (see, for example, \cite{DiscrepancyMinimization-Bansal-FOCS2010} and \cite{DiscrepancyMinimization-LovettMekaFOCS12}), we can compute signs matching this upper bound in polynomial time.
\end{remark}

\begin{remark}
After writing this article the author realized a similar argument in fact appeared in \cite{Spencer1985SixSD} to provide an exponential lower bound on the \textit{number} of distinct Rudin-Shapiro sequences for the case when $p_i$ are monomials.
\end{remark}

\begin{proof}[Proof of Theorem 2]
Indeed for any choice of signs $x_1, \dots, x_n \in \{-1,1\}$, denoting $q := \sum_{i=1}^n x_i T_i$,
\begin{align*}
 \|q\|_\infty = \sup_{\theta \in [0, \pi]} |q(\cos(\theta))| & \ge \sqrt{\frac{1}{\pi} \int_0^\pi q^2 (\cos(\theta)) \ \mathrm{d}\theta} \\ & \ge \sqrt{\frac{1}{\pi} \cdot \sum_{k=1}^n \underbrace{\int_0^\pi \cos^2 (k \theta) \ \mathrm{d}\theta}_{= \pi/2} + \frac{1}{\pi} \cdot \sum_{k \neq \ell} x_k x_\ell \underbrace{\int_0^\pi \cos(k \theta) \cos(\ell \theta) \ \mathrm{d}\theta}_{=0}} \\ & = \sqrt{n/2}. \qedhere \end{align*}
\end{proof}
\newpage
\begin{proof}[Proof of Theorem 3]
Denote $x_k := \cos(k\pi/9n - \pi/18n)$ for $k \in [n]$,  the roots of $T_{9n}$. For each $p_i$ we associate a vector $\bm{v}_i \in \setR^{9n}$ with $v_{ij} := p_i (x_j)$. In order to obtain an upper bound for the $\ell_2$-norm of these vectors, we make use of the \textit{Gauss-Chebyshev quadrature formula} \cite{NumAnalysis2002} 
\[\int_{0}^\pi p(\cos \theta) \ \mathrm{d} \theta = \frac{\pi}{n} \cdot \sum_{k=1}^n p(x_k),\]
which holds for all polynomials $p$ of degree less than $2n$. For completeness, we include a short proof. First, we claim that it holds for $p$ of degree less than $n$. Indeed, any such polynomial may be written as $p = \sum_{j=0}^{n-1} c_j T_j$ for some $c_j \in \setR$. Then 
\begin{align*} \int_{0}^\pi p(\cos \theta) \ \mathrm{d} \theta & = \int_{0}^\pi c_0 \ \mathrm{d} \theta + \sum_{j=1}^{n-1} c_j \underbrace{\int_{0}^\pi \cos (j \theta) \ \mathrm{d} \theta}_{=0} \\ &= \pi c_0, \end{align*}
and also
\begin{align*} \frac{\pi}{n} \cdot \sum_{k=1}^n p(x_k) & = \frac{\pi}{n} \cdot \sum_{k=1}^n c_0 + \frac{\pi}{n} \cdot \sum_{j=1}^{n-1} c_j \sum_{k=1}^n T_j(x_k) \\ &= \pi c_0 + \frac{\pi}{n} \cdot \sum_{j=1}^{n-1} \mathrm{Re}\Big(e^{-j\pi/2n} \cdot \underbrace{\sum_{k=1}^n e^{jk\pi/n}}_{= 0}\Big) \\ &= \pi c_0, \end{align*} 
as claimed. To see that the formula in fact holds for $p$ of degree less than $2n$, it suffices to divide $p$ by $T_n$ and observe the quotient and remainder both have degrees less than $n$, so that the correctness for $p$ follows from the formula for the remainder term.

Applying the Gauss-Chebyshev quadrature formula to each $p_i^2$ gives $\|\bm{v}_i\|_2 \le \sqrt{9n/\pi}$, thus the Koml\'os conjecture with constant $C$ implies the existence of signs $x_1, \dots, x_n \in \{-1,1\}$ so that $q := \sum_{i=1}^n x_i p_i$ satisfies $|q(x_k)| \le C\sqrt{9n/\pi}$ for $k \in [n]$. Applying Lemma 7 with $\theta_0 := \pi/18n$ yields $\|q\|_\infty < 3C \sqrt{n}$.
\end{proof}
{\Large \textbf{Acknowledgments}}
\\
\\
We would like to thank Anup Rao and Rodrigo Angelo for helpful discussions, Thomas Rothvoss for feedback in early drafts of this work, and the anonymous referee for finding (and fixing) a small error in the proof of Lemma 7.

\bibliographystyle{alpha}
\bibliography{BalancingPolynomials}
 
\end{document}